\documentclass[
  11pt,
  reqno
]{amsart}

\usepackage{mathtools}
\usepackage{amssymb}
\usepackage{amsthm}
\usepackage{mathrsfs}
\usepackage{tikz-cd}
\usepackage[no-math]{fontspec}
\usepackage{microtype}
\usepackage{enumitem}

\usepackage[
  backend=biber,
  style=numeric,
  doi=true,
  sorting=nyt,
  eprint=true,
  giveninits=true,
  maxnames=99
]{biblatex}
\DeclareFieldFormat[article]{title}{\mkbibemph{#1}}
\renewbibmacro*{in:}{%
  \ifentrytype{article}
    {}
    {\printtext{\bibstring{in}\intitlepunct}}}

\usepackage[hidelinks]{hyperref}
\usepackage[nameinlink,noabbrev]{cleveref}

\newcommand{\Z}{\mathbb{Z}}
\newcommand{\R}{\mathbb{R}}
\newcommand{\cF}{\mathcal{F}}
\newcommand{\Lip}{\operatorname{Lip}}
\newcommand{\dist}{\operatorname{dist}}
\newcommand{\pathLength}{\operatorname{len}}

\DeclarePairedDelimiter{\abs}{\lvert}{\rvert}
\DeclarePairedDelimiter{\norm}{\lVert}{\rVert}
\DeclarePairedDelimiter{\set}{\lbrace}{\rbrace}

\newcommand{\closedBall}[3]{B_{#1}(#2,#3)}
\newcommand{\ellOne}{\ell_1}
\newcommand{\ellOneOf}[1]{\ell_1(#1)}
\newcommand{\finitelySupported}[1]{c_{00}(#1)}
\newcommand{\oneSum}[2]{#1\oplus_1 #2}
\newcommand{\inftySum}[2]{#1\oplus_\infty #2}
\newcommand{\freeSpace}[1]{\cF(#1)}
\newcommand{\dirac}[1]{\delta_{#1}}
\newcommand{\freeEmbedding}[1]{\delta_{#1}}
\newcommand{\pointedLip}[1]{\Lip_0(#1)}
\newcommand{\boundedContinuous}[1]{C_b(#1)}
\newcommand{\dual}[1]{#1^*}
\newcommand{\closedLinearSpan}[1]{\overline{\operatorname{span}#1}}

\newcommand{\tripleSource}[1]{M_{#1}}
\newcommand{\tripleTarget}[1]{N_{#1}}
\newcommand{\tripleMap}[1]{V_{#1}}
\newcommand{\protectedSourceEmbedding}[2]{\iota_{#1}^{#2}}
\newcommand{\protectedTargetEmbedding}[2]{\jmath_{#1}^{#2}}
\newcommand{\protectedExtensionRetraction}[2]{R_{#1}^{#2}}
\newcommand{\extendedMap}{V^{+}}
\newcommand{\attachmentSet}{A}
\newcommand{\attachmentMap}{\varphi}
\newcommand{\attachmentInclusion}{\iota_A}
\newcommand{\attachmentGap}{C}
\newcommand{\attachmentHeight}{H}
\newcommand{\sourceEmbedding}{\iota_E}
\newcommand{\targetEmbedding}{\iota_F}
\newcommand{\adjunctionSpace}{P}
\newcommand{\adjunctionInclusion}{q}
\newcommand{\adjunctionTargetEmbedding}{\iota_N}
\newcommand{\adjunctionMetric}{d_P}
\newcommand{\relativeEnvelope}{B_0}
\newcommand{\gluingSubspace}{Z}
\newcommand{\relativeQuotientMap}{\Pi}
\newcommand{\relativeDualFunctional}[2]{\Phi_{#1,#2}}
\newcommand{\targetFunctional}{\ell}
\newcommand{\oldTargetEmbedding}{j_0}
\newcommand{\relativeEmbedding}{\eta}
\newcommand{\quotientSpace}{Q}
\newcommand{\metricQuotientMap}{\pi_Q}
\newcommand{\quotientMetric}{d_Q}
\newcommand{\quotientBasepoint}{o}
\newcommand{\distanceFunction}[1]{d_{#1}}
\newcommand{\kuratowskiEmbedding}{K}
\newcommand{\kuratowskiSpan}{G_Q}
\newcommand{\injectivityCoordinate}{\widetilde{K}}
\newcommand{\combinedEmbedding}{J}
\newcommand{\cutoffMap}{\lambda}
\newcommand{\sourceRetraction}{\rho_E}
\newcommand{\adjunctionRetraction}{\rho_P}
\newcommand{\linearizedAdjunctionRetraction}{\widehat{\rho}_P}
\newcommand{\quotientRetractionLift}{T}
\newcommand{\relativeRetraction}{R_0}
\newcommand{\firstCoordinateProjection}{\pi_{B_0}}
\newcommand{\protectedRetraction}{R}
\newcommand{\minimizingPath}[1]{\mathcal P_{#1}}
\newcommand{\baseComparisonPath}[2]{\mathcal D_{#1,#2}}
\newcommand{\mixedAttachmentComparisonPath}[1]{\mathcal L_{#1}}
\newcommand{\oppositeMixedAttachmentComparisonPath}[1]{\widetilde{\mathcal L}_{#1}}
\newcommand{\mixedPath}{\mathcal R}
\newcommand{\baseRoute}[3]{\mathcal B_{#1,#2,#3}}
\newcommand{\topRoute}[3]{\mathcal T_{#1,#2,#3}}

\newcommand{\continuumCardinal}{\mathfrak c}
\newcommand{\stageCardinal}{\theta}
\newcommand{\recursionCardinal}{\kappa}
\newcommand{\sourceStage}[1]{M_{#1}}
\newcommand{\targetStage}[1]{N_{#1}}
\newcommand{\stageMap}[1]{V_{#1}}
\newcommand{\sourceStageEmbedding}[2]{\iota_{#1}^{#2}}
\newcommand{\targetStageEmbedding}[2]{\jmath_{#1}^{#2}}
\newcommand{\stageRetraction}[2]{R_{#1}^{#2}}
\newcommand{\extensionStages}[1]{I(#1)}
\newcommand{\stageProjection}[1]{P_{#1}}
\newcommand{\coordinateTruncation}[1]{\tau_{#1}}
\newcommand{\stageTailTruncation}[2]{\tau_{#1,#2}}
\newcommand{\stageEnumeration}[1]{e_{#1}}
\newcommand{\scheduleMap}{\sigma}

\newcommand{\bentMap}{S}
\newcommand{\bentTarget}{\ellOne^2}
\theoremstyle{plain}
\newtheorem{theorem}{Theorem}[section]
\newtheorem{lemma}[theorem]{Lemma}
\newtheorem{proposition}[theorem]{Proposition}

\theoremstyle{definition}

\newtheorem{problem}[theorem]{Problem}

\theoremstyle{remark}

\newenvironment{acknowledgements}{%
  \medskip\noindent\textit{Acknowledgements.}\ }{\par}
\newenvironment{useofai}{%
  \medskip\noindent\textbf{Use of AI.}\ }{\par}

\crefname{theorem}{theorem}{theorems}
\crefname{lemma}{lemma}{lemmas}
\crefname{proposition}{proposition}{propositions}
\crefname{corollary}{corollary}{corollaries}
\crefname{definition}{definition}{definitions}
\crefname{example}{example}{examples}
\crefname{problem}{problem}{problems}
\crefname{question}{question}{questions}
\crefname{remark}{remark}{remarks}
\Crefname{theorem}{Theorem}{Theorems}
\Crefname{lemma}{Lemma}{Lemmas}
\Crefname{problem}{Problem}{Problems}

\title{A counterexample to Scottish Book\\Problem 155}
\author{Yoshito Ishiki}
\address{Department of Mathematical Sciences\\
Tokyo Metropolitan University\\
Minami-osawa, Hachioji, Tokyo 192-0397, Japan}
\email{ishiki-yoshito@tmu.ac.jp}
\date{}

\subjclass[2020]{Primary 46B20, Secondary 54E40}
\keywords{Banach space, local isometry, Lipschitz-free space, Scottish Book}
\hypersetup{
  pdftitle={A counterexample to Scottish Book Problem 155},
  pdfauthor={Yoshito Ishiki},
  pdfsubject={Scottish Book Problem 155},
  pdfkeywords={Banach space, local isometry, Lipschitz-free space, Scottish Book}
}

\begin{document}

\begin{abstract}
We construct real Banach spaces $X$ and $Y$ and a bijection $U\colon X\to Y$ which preserves all pairwise distances on every closed ball of radius $1/4$ but is not a global isometry. The proof starts from an explicit injective map that preserves every distance at most $1/2$ and shortens one longer distance. We then enlarge the source and target through a transfinite recursion that inserts every missing target point while preserving both injectivity and this distance threshold.
\end{abstract}

\maketitle

\section{Introduction}
Scottish Book Problem 155 asks whether local preservation of distances forces global preservation for a bijection between real Banach spaces
\cite[Problem~155]{mauldin2015}.
We use the closed ball convention from the problem.
Thus,
for a real Banach space
$X$,
we write
$\closedBall{X}{x}{r}=\set{z\in X\mid \norm{z-x}_{X}\le r}$.
The hypothesis says that for every
$x\in X$
there is
$r_x>0$
such that a bijection
$U\colon X\to Y$
preserves all pairwise distances on
$\closedBall{X}{x}{r_x}$.
The question is whether
$U$
must preserve every distance in
$X$.

Mori proved that the answer is positive when the source space is separable,
even when injectivity is replaced by surjectivity
\cite[Theorem~3]{mori2024}.
Basso obtained a positive answer under the stronger metric notion of a local isometry
\cite[Corollary~1.3]{basso2023}.
In that notion the local restriction must also map onto an open subset of the target.
In particular,
the result of Basso applies to the original hypothesis when the inverse bijection is continuous.
For broader historical context on Scottish Book Problem~155,
see
\cite[Problem~155]{domoradzki2026}.

The unrestricted statement fails.

\begin{theorem}\label{thm:counterexample}
There are real Banach spaces
$X$
and
$Y$
and a bijection
$U\colon X\to Y$
such that
\[
\norm{U(p)-U(q)}_{Y}=\norm{p-q}_{X}
\]
for every
$x\in X$
and every
$p,q\in \closedBall{X}{x}{1/4}$,
while
$U$
is not a global isometry.
\end{theorem}

The construction begins with an explicit piecewise linear map
$\bentMap\colon\R\to\bentTarget$.
It is injective and preserves every distance at most
$1/2$
but contracts one distant pair.
The main difficulty is to make its image equal to an entire Banach space without losing injectivity or the fixed short distance scale.

The successor mechanism is
\Cref{lem:protected-extension}.
Starting from an injective map
$V\colon M\to N$
and a missing point
$y\in N\setminus V(M)$,
it enlarges the source and target Banach spaces to
$E=\oneSum{M}{\R}$
and
$F$,
and constructs a map
$\extendedMap\colon E\to F$.
The new map puts
$y$
into its image
and retains every distance up to the prescribed scale.
Its injectivity coordinate is obtained from the Kuratowski embedding of the metric quotient formed by collapsing the old target
$N$.
The lemma also constructs a 1-Lipschitz linear retraction
$\protectedRetraction\colon F\to N$
onto the old target.
The maps satisfy
\[
(\protectedRetraction\circ\extendedMap)(m,s)=V(m)
\]
whenever
$\abs{s}\le L$.
This coordinate recovery identity prevents distinct points from acquiring the same image when increasing unions are completed at limit ordinals.

After the successor and limit mechanisms are assembled,
a recursion of regular cardinal length schedules every target point for entry into the image.
The final unions are already complete,
and the resulting map is bijective.
The original contracted pair survives through linear isometric inclusions.

\medskip
\noindent\textbf{Organization.}
In
\Cref{sec:metric-linear-tools},
we collect the fixed short scale estimate and the facts about Lipschitz-free spaces used in the sequel.
In
\Cref{sec:protected-extension},
we establish the protected one point extension which forms the successor mechanism.
In
\Cref{sec:coherent-limit-stages},
we show how coherent 1-Lipschitz retractions with this coordinate recovery property preserve injectivity after completion at limit ordinals.
In
\Cref{sec:counterexample},
we construct the bent map and carry out the regular cardinal recursion to prove
\Cref{thm:counterexample}
and formulate
\Cref{prob:sufficient-conditions}.

\medskip
\noindent\textbf{Conventions and notation.}
All Banach spaces in this paper are real.
An embedding between Banach spaces is linear and isometric.

\begin{useofai}
An earlier version of this work was publicly posted as an
\emph{AI-Assisted Research Report}
\cite{ishiki2026aireport155}.
The present paper is a revised version of that report.
OpenAI Codex was used in the preparation of this manuscript for language editing,
\LaTeX{}
typesetting assistance,
literature searches,
and exploration of proof constructions.
In particular,
Codex suggested the counterexample construction and assisted in developing the protected one point extension and the coherent transfinite recursion.
The author verified the arguments and takes full responsibility for the mathematical content and the final version of the manuscript.
The author also understands all proofs and constructions presented in this paper.
\end{useofai}

\begin{acknowledgements}
The author was supported by JSPS KAKENHI Grant Number JP24KJ0182.
\end{acknowledgements}

\section{Metric and linear tools}\label{sec:metric-linear-tools}
We first pass from preservation at one fixed short scale to a global 1-Lipschitz estimate.
We then recall the Lipschitz-free construction which turns the metric maps used later into linear maps between Banach spaces.

\subsection{Maps which preserve a fixed short scale}

A map
$T\colon X\to Y$
between metric spaces is 1-Lipschitz if
\[
d_Y(T(x_0),T(x_1))\le d_X(x_0,x_1)
\]
for all
$x_0,x_1\in X$.
For a linear map between Banach spaces,
this is equivalent to having operator norm at most one.

Let
$M$
and
$N$
be Banach spaces,
let
$r>0$,
and let
$V\colon M\to N$
satisfy
\begin{equation}\label{eq:short-scale}
\norm{V(m_0)-V(m_1)}_{N}=\norm{m_0-m_1}_{M}
\quad
\text{whenever }
\norm{m_0-m_1}_{M}\le r.
\end{equation}
Fix
$m_0,m_1\in M$.
Put
$d=\norm{m_0-m_1}_{M}$,
choose a positive integer
$n$
such that
$d/n\le r$,
and set
\[
m_k=m_0+\frac{k}{n}(m_1-m_0)
\]
for
$0\le k\le n$.
Each consecutive pair has distance
$d/n$,
so
\eqref{eq:short-scale}
applies to every pair
$m_{k-1},m_k$.
The triangle inequality in
$N$
therefore yields
\begin{equation}\label{eq:global-one-lipschitz}
\begin{aligned}
\norm{V(m_0)-V(m_1)}_{N}
&\le \sum_{k=1}^{n}\norm{V(m_k)-V(m_{k-1})}_{N}\\
&=\sum_{k=1}^{n}\norm{m_k-m_{k-1}}_{M}\\
&=\norm{m_0-m_1}_{M}
\end{aligned}
\end{equation}
for all
$m_0,m_1\in M$.
Thus every map satisfying
\eqref{eq:short-scale}
is globally 1-Lipschitz.

The successor construction uses both the sum norm and the maximum norm.
We write
$\oneSum{M}{E}$
and
$\inftySum{M}{E}$
for direct sums with the sum norm and the maximum norm,
respectively.

We shall use the following norming form of the Hahn--Banach theorem
\cite[Theorem~4.4.2]{geiss-geiss2025}.
For every real Banach space
$X$
and every
$x\in X$,
there is
$\phi\in\dual{X}$
such that
\[
\norm{\phi}_{\dual{X}}\le 1
\qquad\text{and}\qquad
\phi(x)=\norm{x}_{X}.
\]
A functional with these properties is called a norming functional for
$x$.

\subsection{Lipschitz-free spaces}

We use Lipschitz-free spaces to build a Banach space from the glued metric space constructed in
\Cref{sec:protected-extension}
and to turn a 1-Lipschitz retraction into a linear map.
Let
$(P,d,0)$
be a pointed metric space.
Let
$\pointedLip{P}$
be the Banach space of real Lipschitz functions
$f\colon P\to\R$
which satisfy
$f(0)=0$,
with norm
$\norm{f}_{\pointedLip{P}}=\Lip(f)$.
For
$p\in P$,
let
$\dirac{p}\in\dual{\pointedLip{P}}$
be the evaluation functional defined by
\[
\dirac{p}(f)=f(p).
\]
The Lipschitz-free Banach space over
$P$
is
\[
\freeSpace{P}
=\closedLinearSpan{\set{\dirac{p}\mid p\in P}}
\subset \dual{\pointedLip{P}}.
\]
Equivalently,
$\freeSpace{P}$
is the completion of the finitely supported signed measures on
$P$
with total mass zero under the Kantorovich--Rubinstein norm.
For
$\mu=\mu^+-\mu^-$,
this norm is the Wasserstein 1 distance between the equal mass measures
$\mu^+$
and
$\mu^-$.

Denote the canonical map by
\[
\freeEmbedding{P}\colon P\to \freeSpace{P},
\qquad
\freeEmbedding{P}(p)=\dirac{p},
\]
where
$\dirac{0}=0$
and
\begin{equation}\label{eq:free-isometry}
\norm{\dirac{p}-\dirac{q}}_{\freeSpace{P}}=d(p,q).
\end{equation}
The dual space
$\dual{\freeSpace{P}}$
is canonically identified with
$\pointedLip{P}$.
Every pointed 1-Lipschitz map
$T\colon P\to X$
to a Banach space has a unique 1-Lipschitz linear map
$\widehat{T}\colon\freeSpace{P}\to X$
such that
$\widehat{T}(\dirac{p})=T(p)$
for every
$p\in P$.
Equivalently,
$\widehat{T}\circ\freeEmbedding{P}=T$.
Arens and Eells construct the molecule space for an arbitrary metric space and prove that its canonical map is isometric
\cite[Section~2(C) and Section~3]{arens-eells1956}.
For a Banach space regarded as a pointed metric space,
Godefroy and Kalton record the canonical predual,
the canonical isometry,
and the linearization property
\cite[Definition~1.1,
Proposition~2.1,
and Lemma~2.5]{godefroy-kalton2003}.
Michael later obtained a short proof of the isometric embedding in
\eqref{eq:free-isometry}
using evaluation functionals and distance functions
\cite[Theorem]{michael1964}.

We also use the real McShane extension theorem.
A real Lipschitz function on a subset of a metric space extends to the whole space with the same Lipschitz constant
\cite[Theorem~1]{mcshane1934}.

\section{A protected one point extension}\label{sec:protected-extension}
We separate the successor mechanism into its metric,
linear,
and recovery components.
We first name the common interface used by the three components.
Fix
$r>0$.
A triple
\[
\mathcal A
=\bigl(
\tripleSource{\mathcal A},
\tripleTarget{\mathcal A},
\tripleMap{\mathcal A}
\bigr)
\]
consists of real Banach spaces
$\tripleSource{\mathcal A}$
and
$\tripleTarget{\mathcal A}$
and a map
\[
\tripleMap{\mathcal A}
\colon
\tripleSource{\mathcal A}
\to
\tripleTarget{\mathcal A}.
\]
We call
$\mathcal A$
a uniformly locally isometric triple at scale
$r$
if
$\tripleMap{\mathcal A}$
is injective and
\begin{equation}\label{eq:uniform-local-isometry-triple}
\norm{
\tripleMap{\mathcal A}(m_0)
-\tripleMap{\mathcal A}(m_1)
}_{\tripleTarget{\mathcal A}}
=
\norm{m_0-m_1}_{\tripleSource{\mathcal A}}
\end{equation}
whenever
\[
\norm{m_0-m_1}_{\tripleSource{\mathcal A}}
\le r.
\]

Let
$L>0$,
let
$\mathcal A$
be a uniformly locally isometric triple at scale
$r$,
and let
\[
y\in
\tripleTarget{\mathcal A}
\setminus
\tripleMap{\mathcal A}(\tripleSource{\mathcal A}).
\]
For another triple
\[
\mathcal B
=\bigl(
\tripleSource{\mathcal B},
\tripleTarget{\mathcal B},
\tripleMap{\mathcal B}
\bigr),
\]
consider maps
\[
\protectedSourceEmbedding{\mathcal A}{\mathcal B}
\colon
\tripleSource{\mathcal A}
\to
\tripleSource{\mathcal B},
\qquad
\protectedTargetEmbedding{\mathcal A}{\mathcal B}
\colon
\tripleTarget{\mathcal A}
\to
\tripleTarget{\mathcal B},
\]
and
\[
\protectedExtensionRetraction{\mathcal A}{\mathcal B}
\colon
\tripleTarget{\mathcal B}
\to
\tripleTarget{\mathcal A}.
\]
We say that
$\mathcal B$,
together with these maps,
is an
$L$-protected extension of
$\mathcal A$
at
$y$
if the first two maps are linear isometric embeddings,
the third map is linear and 1-Lipschitz,
and the following conditions hold.
\begin{enumerate}[label=\textup{(E\arabic*)},ref=\textup{(E\arabic*)}]
\item\label{item:extension-scale}
The triple
$\mathcal B$
is uniformly locally isometric at scale
$r$.
\item\label{item:extension-domain}
We have
\[
\tripleSource{\mathcal B}
=\oneSum{\tripleSource{\mathcal A}}{\R}
\]
and
\[
\protectedSourceEmbedding{\mathcal A}{\mathcal B}(m)
=(m,0)
\]
for
$m\in\tripleSource{\mathcal A}$.
\item\label{item:extension-commutation}
The extension square commutes,
which means that
\[
\tripleMap{\mathcal B}
\circ
\protectedSourceEmbedding{\mathcal A}{\mathcal B}
=
\protectedTargetEmbedding{\mathcal A}{\mathcal B}
\circ
\tripleMap{\mathcal A}.
\]
\item\label{item:extension-hit}
The embedded point
\[
\protectedTargetEmbedding{\mathcal A}{\mathcal B}(y)
\]
belongs to
\[
\tripleMap{\mathcal B}(\tripleSource{\mathcal B}).
\]
\item\label{item:extension-retraction}
The map
$\protectedExtensionRetraction{\mathcal A}{\mathcal B}$
is a retraction onto the old target in the sense that
\[
\protectedExtensionRetraction{\mathcal A}{\mathcal B}
\circ
\protectedTargetEmbedding{\mathcal A}{\mathcal B}
=
\operatorname{id}_{\tripleTarget{\mathcal A}}.
\]
\item\label{item:extension-recovery}
For
$m\in\tripleSource{\mathcal A}$
and
$\abs{s}\le L$,
we have
\[
\bigl(
\protectedExtensionRetraction{\mathcal A}{\mathcal B}
\circ
\tripleMap{\mathcal B}
\bigr)(m,s)
=
\tripleMap{\mathcal A}(m).
\]
\end{enumerate}

The next lemma constructs the metric adjunction and records the properties needed for its later linearization.

\begin{lemma}\label{lem:metric-adjunction}
Let
$M$
and
$N$
be real Banach spaces,
let
$r>0$,
and let
$V\colon M\to N$
be injective and satisfy
\eqref{eq:short-scale}.
Fix
$y\in N\setminus V(M)$
and
$a\in M$.
Put
$\attachmentGap=\norm{V(a)-y}_{N}$,
choose
$\attachmentHeight>\attachmentGap+4r$,
and set
\[
E=\oneSum{M}{\R},
\qquad
\sourceEmbedding\colon M\to E,
\qquad
\sourceEmbedding(m)=(m,0).
\]
Define
\[
\attachmentSet=(M\times\set{0})\cup\set{(a,\attachmentHeight)}
\]
and
$\attachmentMap\colon\attachmentSet\to N$
by
\[
\attachmentMap(m,0)=V(m),
\qquad
\attachmentMap(a,\attachmentHeight)=y.
\]
Let
$\adjunctionSpace=N\cup_{\attachmentMap}E$
be the quotient of the disjoint union of
$N$
and
$E$
obtained by identifying every
$u\in\attachmentSet$
with
$\attachmentMap(u)$.
Let
\[
\adjunctionInclusion\colon E\to\adjunctionSpace
\quad\text{and}\quad
\adjunctionTargetEmbedding\colon N\to\adjunctionSpace
\]
be the canonical maps.
For
$Z\in\set{E,N}$
and
$u,v\in Z$,
write
$[u,v]_Z$
for the oriented affine segment from
$u$
to
$v$.
We write
$\mathcal C_0*\mathcal C_1$
for concatenation of compatible paths and
$\overline{\mathcal C_0}$
for the reversed path.
An admissible polygonal path
$\mathcal C$
in
$\adjunctionSpace$
is a finite concatenation of images under
$\adjunctionInclusion$
or
$\adjunctionTargetEmbedding$
of affine line segments lying in
$E$
or
$N$.
If
$\mathcal C$
has
$k$
pieces with endpoints
$u_j,v_j\in Z_j$
for
$1\le j\le k$,
where
$Z_j\in\set{E,N}$,
define
\[
\pathLength(\mathcal C)
=\sum_{j=1}^{k}\norm{u_j-v_j}_{Z_j}.
\]
The path metric on
$\adjunctionSpace$
is
\[
\adjunctionMetric(z,w)
=\inf\set{
\pathLength(\mathcal C)
\mid
\mathcal C\text{ is an admissible polygonal path from }z\text{ to }w
}.
\]
Then
$\attachmentMap$
is injective and 1-Lipschitz,
$\adjunctionTargetEmbedding$
is an isometric embedding,
and
$\adjunctionInclusion$
is injective and 1-Lipschitz.
The maps satisfy
\[
\adjunctionInclusion\circ\sourceEmbedding
=\adjunctionTargetEmbedding\circ V
\quad\text{and}\quad
\adjunctionInclusion(a,\attachmentHeight)
=\adjunctionTargetEmbedding(y).
\]
Moreover,
if
$x_0,x_1\in E$
and
$\norm{x_0-x_1}_{E}\le r$,
then
\[
\adjunctionMetric(\adjunctionInclusion(x_0),\adjunctionInclusion(x_1))
=\norm{x_0-x_1}_{E}.
\]
\end{lemma}

\begin{proof}
The set
$\attachmentSet$
is closed in
$E$.
The set
$\attachmentSet$
consists of the points of
$E$
that will be identified with their images under
$\attachmentMap$
in the construction below.
The map
$\attachmentMap$
is injective because
$V$
is injective and
$y\notin V(M)$.
Its restriction to
$M\times\set{0}$
is 1-Lipschitz by
\eqref{eq:global-one-lipschitz}.
For
$m\in M$,
the distance between
$(m,0)$
and the remaining point
$(a,\attachmentHeight)$
satisfies
\[
\begin{split}
\norm{V(m)-y}_{N}
&\le \norm{V(m)-V(a)}_{N}+\attachmentGap
\le \norm{m-a}_{M}+\attachmentGap\\
&<\norm{m-a}_{M}+\attachmentHeight
=\norm{(m,0)-(a,\attachmentHeight)}_{E}.
\end{split}
\]
Thus
$\attachmentMap$
is 1-Lipschitz on
$\attachmentSet$.

The definition of
$\adjunctionMetric$
makes
$\adjunctionInclusion$
1-Lipschitz.
Let
$\attachmentInclusion\colon\attachmentSet\to E$
be the inclusion.
The canonical maps satisfy
$\adjunctionInclusion\circ\attachmentInclusion
=\adjunctionTargetEmbedding\circ\attachmentMap$,
as displayed in the adjunction square
\begin{equation}\label{diag:metric-adjunction}
\begin{tikzcd}[column sep=large,row sep=large]
\attachmentSet
  \arrow[r,hook,"\attachmentInclusion"]
  \arrow[d,"\attachmentMap"']
& E \arrow[d,"\adjunctionInclusion"]\\
N \arrow[r,"\adjunctionTargetEmbedding"']
& \adjunctionSpace.
\end{tikzcd}
\end{equation}
Evaluating the commuting identity on
$(m,0)$
and
$(a,\attachmentHeight)$
yields the two identities in the statement.
For readability,
we write
$n$
for
$\adjunctionTargetEmbedding(n)$
inside
$\adjunctionSpace$.

Every admissible polygonal path with both endpoints in
$E$
may be shortened so that it is either one segment in
$E$
or leaves
$E$
once and returns once.
To see this,
consider an internal segment
$[u,v]_E$
whose endpoints belong to
$\attachmentSet$.
The corresponding segment in
$N$
satisfies
\[
\begin{aligned}
\pathLength([\attachmentMap(u),\attachmentMap(v)]_N)
&=\norm{\attachmentMap(u)-\attachmentMap(v)}_{N}\\
&\le \norm{u-v}_{E}
=\pathLength([u,v]_E),
\end{aligned}
\]
because
$\attachmentMap$
is 1-Lipschitz.
Replace every internal
$E$
segment in this way.
Consecutive segments in
$N$
can then be replaced by the affine segment joining their endpoints,
whose length is no larger by the triangle inequality in
$N$.
If the original path never enters
$N$,
the triangle inequality in
$E$
replaces it by the affine segment joining its endpoints.
This gives the asserted two forms without increasing length.
Consequently,
for
$x_0,x_1\in E$,
\begin{equation}\label{eq:adjunction-e-distance}
\begin{split}
&\adjunctionMetric(\adjunctionInclusion(x_0),\adjunctionInclusion(x_1))\\
&\quad=\min\biggl\{
\norm{x_0-x_1}_{E},\\
&\qquad\inf_{u,v\in \attachmentSet}
\bigl(
\norm{x_0-u}_{E}
+\norm{\attachmentMap(u)-\attachmentMap(v)}_{N}
+\norm{v-x_1}_{E}
\bigr)
\biggr\}.
\end{split}
\end{equation}
The second term on the right hand side of
\eqref{eq:adjunction-e-distance}
is the infimum of the lengths of the polygonal paths
$\mathcal C_{u,v}$
defined by
\[
\mathcal C_{u,v}
=[x_0,u]_E
*[\attachmentMap(u),\attachmentMap(v)]_N
*[v,x_1]_E.
\]
Thus
\[
\pathLength(\mathcal C_{u,v})
=\norm{x_0-u}_{E}
+\norm{\attachmentMap(u)-\attachmentMap(v)}_{N}
+\norm{v-x_1}_{E}.
\]
For
$n\in N$,
the corresponding formula is
\begin{equation}\label{eq:adjunction-n-distance}
\adjunctionMetric(\adjunctionInclusion(x_0),n)
=\inf_{u\in \attachmentSet}
\bigl(
\norm{x_0-u}_{E}+\norm{\attachmentMap(u)-n}_{N}
\bigr).
\end{equation}

The same path replacement shows that
$\adjunctionTargetEmbedding$
is an isometric embedding.
The map
$\adjunctionInclusion$
is injective.
Indeed,
if
$\adjunctionMetric(\adjunctionInclusion(x_0),\adjunctionInclusion(x_1))=0$,
then
\eqref{eq:adjunction-e-distance}
shows that either
$\norm{x_0-x_1}_{E}=0$
or the second term on its right hand side is zero.
In the latter case,
there are
$u_k,v_k\in \attachmentSet$
such that
\[
\norm{x_0-u_k}_{E}\to 0,
\qquad
\norm{x_1-v_k}_{E}\to 0,
\qquad
\norm{\attachmentMap(u_k)-\attachmentMap(v_k)}_{N}\to 0.
\]
Closedness of
$\attachmentSet$
places
$x_0$
and
$x_1$
in
$\attachmentSet$.
Continuity and injectivity of
$\attachmentMap$
then imply
$x_0=x_1$.
Thus
$x_0=x_1$
in either case.

We next show that
$\adjunctionInclusion$
preserves every distance at most
$r$.
Write
\[
x_0=(m_0,s_0),
\qquad
x_1=(m_1,s_1),
\qquad
d=\norm{x_0-x_1}_{E}\le r.
\]
For the two points of
$\attachmentSet$
in the second term of
\eqref{eq:adjunction-e-distance},
there are three possibilities.
Put
\[
t=(a,\attachmentHeight).
\]
\begin{enumerate}[label=\textup{(A\arabic*)},ref=\textup{(A\arabic*)}]
\item\label{item:two-base-attachments}
If both points belong to
$M\times\set{0}$,
write them as
$(u,0)$
and
$(v,0)$.
Since
$\norm{m_0-m_1}_{M}\le d\le r$,
define the comparison path in
$N$
by
\[
\baseComparisonPath{u}{v}
=[V(m_0),V(u)]_N
*[V(u),V(v)]_N
*[V(v),V(m_1)]_N.
\]
Its length satisfies
\begin{align*}
\norm{m_0-m_1}_{M}
&=\norm{V(m_0)-V(m_1)}_{N}\\
&\le \pathLength(\baseComparisonPath{u}{v})\\
&=\norm{V(m_0)-V(u)}_{N}
+\norm{V(u)-V(v)}_{N}
+\norm{V(v)-V(m_1)}_{N}\\
&\le \norm{m_0-u}_{M}
+\norm{V(u)-V(v)}_{N}
+\norm{v-m_1}_{M}.
\end{align*}
Here the equality uses
\eqref{eq:short-scale},
the first inequality is the triangle inequality along
$\baseComparisonPath{u}{v}$,
and the final inequality uses
\eqref{eq:global-one-lipschitz}.
The corresponding path in
$\adjunctionSpace$
is
\[
\mathcal C_{(u,0),(v,0)}
=[x_0,(u,0)]_E
*[V(u),V(v)]_N
*[(v,0),x_1]_E.
\]
Adding its three segment lengths gives
\begin{align*}
\pathLength(\mathcal C_{(u,0),(v,0)})
&=\norm{x_0-(u,0)}_{E}
+\norm{V(u)-V(v)}_{N}
+\norm{(v,0)-x_1}_{E}\\
&=\norm{m_0-u}_{M}+\abs{s_0}
+\norm{V(u)-V(v)}_{N}
+\norm{v-m_1}_{M}+\abs{s_1}\\
&\ge \norm{m_0-m_1}_{M}+\abs{s_0-s_1}\\
&=\norm{x_0-x_1}_{E}=d.
\end{align*}
\item\label{item:two-top-attachments}
If both points equal
$t$,
the corresponding path is
\[
\mathcal C_{t,t}
=[x_0,t]_E*[y,y]_N*[t,x_1]_E.
\]
The middle segment is degenerate.
Consequently,
the triangle inequality in
$E$
gives
\begin{align*}
\pathLength(\mathcal C_{t,t})
&=\norm{x_0-t}_{E}
+\pathLength([y,y]_N)
+\norm{t-x_1}_{E}\\
&=\norm{x_0-t}_{E}+\norm{t-x_1}_{E}\\
&\ge \norm{x_0-x_1}_{E}=d.
\end{align*}
\item\label{item:mixed-attachments}
If one point belongs to
$M\times\set{0}$
and the other equals
$t$,
write the first point as
$(u,0)$.
Suppose first that the two points occur in the order
$(u,0),t$.
The corresponding path in
$\adjunctionSpace$
is
\[
\mathcal C_{(u,0),t}
=[x_0,(u,0)]_E
*[V(u),y]_N
*[t,x_1]_E.
\]
To estimate its two
$E$
segments,
consider the comparison path in
$E$
defined by
\[
\mixedAttachmentComparisonPath{u}
=[(u,0),x_0]_E*[x_0,x_1]_E*[x_1,t]_E.
\]
This path joins
$(u,0)$
to
$t$.
Since
\[
\norm{(u,0)-t}_{E}
=\norm{u-a}_{M}+\attachmentHeight
\ge \attachmentHeight,
\]
the triangle inequality along
$\mixedAttachmentComparisonPath{u}$
gives
\begin{align*}
\attachmentHeight
&\le \norm{(u,0)-t}_{E}\\
&\le \pathLength(\mixedAttachmentComparisonPath{u})\\
&=\norm{x_0-(u,0)}_{E}
+\norm{x_0-x_1}_{E}
+\norm{t-x_1}_{E}\\
&=\norm{x_0-(u,0)}_{E}+d+\norm{t-x_1}_{E}.
\end{align*}
Therefore,
\[
\norm{x_0-(u,0)}_{E}+\norm{t-x_1}_{E}
\ge \attachmentHeight-d.
\]
Adding the three segment lengths of
$\mathcal C_{(u,0),t}$
and using the nonnegativity of the middle length now yields
\begin{align*}
\pathLength(\mathcal C_{(u,0),t})
&=\norm{x_0-(u,0)}_{E}
+\norm{V(u)-y}_{N}
+\norm{t-x_1}_{E}\\
&\ge \attachmentHeight-d.
\end{align*}
Since
$\attachmentHeight>2r\ge 2d$,
this lower bound is greater than
$d$.
When the two attachment points occur in the opposite order,
the corresponding path is
\[
\mathcal C_{t,(u,0)}
=[x_0,t]_E
*[y,V(u)]_N
*[(u,0),x_1]_E.
\]
In this order,
use the comparison path
\[
\oppositeMixedAttachmentComparisonPath{u}
=[t,x_0]_E*[x_0,x_1]_E*[x_1,(u,0)]_E.
\]
It joins
$t$
to
$(u,0)$,
so the same triangle inequality gives
\begin{align*}
\attachmentHeight
&\le \norm{t-(u,0)}_{E}\\
&\le \pathLength(\oppositeMixedAttachmentComparisonPath{u})\\
&=\norm{x_0-t}_{E}+d+\norm{(u,0)-x_1}_{E}.
\end{align*}
It follows that
\begin{align*}
\pathLength(\mathcal C_{t,(u,0)})
&=\norm{x_0-t}_{E}
+\norm{y-V(u)}_{N}
+\norm{(u,0)-x_1}_{E}\\
&\ge \attachmentHeight-d>d.
\end{align*}
\end{enumerate}
The cases
\ref{item:two-base-attachments}--\ref{item:mixed-attachments}
exhaust the second term in
\eqref{eq:adjunction-e-distance}.
It follows from
\eqref{eq:adjunction-e-distance}
that
\begin{equation}\label{eq:q-short-isometry}
\adjunctionMetric(\adjunctionInclusion(x_0),\adjunctionInclusion(x_1))=\norm{x_0-x_1}_{E}
\end{equation}
whenever
$\norm{x_0-x_1}_{E}\le r$.
\end{proof}

The next lemma turns the metric coordinate
$\adjunctionInclusion$
into a coordinate taking values in a Banach space without changing the required short distances.

\begin{lemma}\label{lem:relative-envelope}
Retain the hypotheses and notation of
\Cref{lem:metric-adjunction}.
Use
$\adjunctionTargetEmbedding$
to regard
$N$
as a subset of
$\adjunctionSpace$,
and take
$0\in N\subset\adjunctionSpace$
as the distinguished basepoint of
$\adjunctionSpace$.
Let
$\gluingSubspace$
be the closed linear span of
$\set{(n,-\dirac{n})\mid n\in N}$
in
$\oneSum{N}{\freeSpace{\adjunctionSpace}}$,
and set
\[
\relativeEnvelope=
\bigl(
\oneSum{N}{\freeSpace{\adjunctionSpace}}
\bigr)/\gluingSubspace,
\]
with quotient map
$\relativeQuotientMap\colon
\oneSum{N}{\freeSpace{\adjunctionSpace}}\to\relativeEnvelope$.
Define
\[
\oldTargetEmbedding\colon N\to\relativeEnvelope,
\qquad
\oldTargetEmbedding(n)=\relativeQuotientMap(n,0),
\]
and
\[
\relativeEmbedding\colon\adjunctionSpace\to\relativeEnvelope,
\qquad
\relativeEmbedding(p)=\relativeQuotientMap(0,\dirac{p}).
\]
Then
$\oldTargetEmbedding$
is a linear isometric embedding,
$\relativeEmbedding$
is 1-Lipschitz,
and
\begin{equation}\label{eq:relative-maps-agree}
\relativeEmbedding\circ\adjunctionTargetEmbedding
=\oldTargetEmbedding.
\end{equation}
Moreover,
if
$x_0,x_1\in E$
and
$\norm{x_0-x_1}_{E}\le r$,
then
\begin{equation}\label{eq:relative-envelope-short}
\norm{\relativeEmbedding(\adjunctionInclusion(x_0))-\relativeEmbedding(\adjunctionInclusion(x_1))}_{\relativeEnvelope}
=\norm{x_0-x_1}_{E}.
\end{equation}
\end{lemma}

\begin{proof}
The dual of the space before taking the quotient is
\[
\bigl(\oneSum{N}{\freeSpace{\adjunctionSpace}}\bigr)^*
=\inftySum{\dual{N}}{\pointedLip{\adjunctionSpace}}.
\]
For
$\targetFunctional\in\dual{N}$
and
$f\in\pointedLip{\adjunctionSpace}$,
define the linear functional
\begin{equation}\label{eq:relative-dual-action}
\begin{aligned}
\relativeDualFunctional{\targetFunctional}{f}
&\colon\oneSum{N}{\freeSpace{\adjunctionSpace}}\to\R,
\\
\relativeDualFunctional{\targetFunctional}{f}(n,\mu)
&=\targetFunctional(n)+\mu(f).
\end{aligned}
\end{equation}
This functional annihilates
$\gluingSubspace$
exactly when
\[
0=\relativeDualFunctional{\targetFunctional}{f}(n,-\dirac{n})
=\targetFunctional(n)-f(n)
\]
for every
$n\in N$.
Thus the annihilator condition is precisely
$f|_N=\targetFunctional$.
Whenever this condition holds,
we use the same symbol for the induced functional on
$\relativeEnvelope$.
In particular,
\begin{equation}\label{eq:relative-dual-quotient-action}
\relativeDualFunctional{\targetFunctional}{f}(\relativeQuotientMap(n,\mu))
=\targetFunctional(n)+\mu(f).
\end{equation}
The dual unit ball of
$\relativeEnvelope$
therefore consists of the pairs
$(\targetFunctional,f)$
satisfying
\[
\norm{\targetFunctional}_{\dual{N}}\le 1,
\qquad
f\in\pointedLip{\adjunctionSpace},
\qquad
\Lip(f)\le 1,
\qquad
f|_N=\targetFunctional.
\]
For
$n\in N$,
the Hahn--Banach theorem provides
$\targetFunctional\in\dual{N}$
such that
$\norm{\targetFunctional}_{\dual{N}}\le 1$
and
$\targetFunctional(n)=\norm{n}_{N}$.
Thus
$\targetFunctional$
is a norming functional for
$n$.
Extend
$\targetFunctional$
to
$\adjunctionSpace$
by the McShane theorem.
If
$f$
denotes this extension,
then
\eqref{eq:relative-dual-quotient-action}
implies
\[
\relativeDualFunctional{\targetFunctional}{f}(\oldTargetEmbedding(n))
=\targetFunctional(n)
=\norm{n}_{N}.
\]
The quotient map yields the reverse estimate because
\[
\begin{aligned}
\norm{\oldTargetEmbedding(n)}_{\relativeEnvelope}
&=\norm{\relativeQuotientMap(n,0)}_{\relativeEnvelope}\\
&=\inf_{z\in\gluingSubspace}
\norm{(n,0)+z}_{\oneSum{N}{\freeSpace{\adjunctionSpace}}}\\
&\le \norm{(n,0)}_{\oneSum{N}{\freeSpace{\adjunctionSpace}}}\\
&=\norm{n}_{N}.
\end{aligned}
\]
The inequality uses
$0\in\gluingSubspace$.
Hence
$\oldTargetEmbedding\colon N\to \relativeEnvelope$
is a linear isometry.
The map
$\relativeEmbedding$
is 1-Lipschitz because
$\relativeQuotientMap$
is 1-Lipschitz and
$\freeEmbedding{\adjunctionSpace}$
is an isometric embedding.
The quotient relation also yields
\eqref{eq:relative-maps-agree}.

We claim that
$\relativeEmbedding\circ\adjunctionInclusion$
preserves every distance at most
$r$.
Fix
$x_0=(m_0,s_0)$
and
$x_1=(m_1,s_1)$
with
\[
d=\norm{x_0-x_1}_{E}\le r,
\qquad
p_i=\adjunctionInclusion(x_i),
\qquad
a_i=\adjunctionMetric(p_i,N).
\]
By
\eqref{eq:q-short-isometry},
we also have
$d=\adjunctionMetric(p_0,p_1)$.

Assume first that
\[
d\le a_0+a_1.
\]
Set
\[
c_0=\min\set{d,a_0}
\qquad\text{and}\qquad
c_1=c_0-d.
\]
Then
$0\le c_0\le a_0$
and
\[
\abs{c_1}=d-c_0=\max\set{d-a_0,0}\le a_1
\]
because
$d\le a_0+a_1$.
Thus
$\abs{c_i}\le a_i$
for
$i\in\set{0,1}$,
and
$c_0-c_1=d$.
Define
$g\colon N\cup\set{p_0,p_1}\to\R$
by
$g|_N=0$
and
$g(p_i)=c_i$.
The function
$g$
is 1-Lipschitz.
Indeed,
for
$n\in N$,
we have
\[
\abs{g(p_i)-g(n)}=\abs{c_i}
\le a_i
\le\adjunctionMetric(p_i,n),
\]
while
\[
\abs{g(p_0)-g(p_1)}=d=\adjunctionMetric(p_0,p_1).
\]
Retain the symbol
$g$
for its McShane extension to
$\adjunctionSpace$.
The functional
$\relativeDualFunctional{0}{g}$
belongs to the dual unit ball and satisfies
\[
\relativeDualFunctional{0}{g}
\bigl(\relativeEmbedding(p_0)-\relativeEmbedding(p_1)\bigr)
=g(p_0)-g(p_1)
=d.
\]
It therefore yields the required lower bound for the
$\relativeEnvelope$
norm of the selected pair.

Assume now that
\begin{equation}\label{eq:collar-case}
d>a_0+a_1.
\end{equation}
By
\eqref{eq:adjunction-n-distance},
we have
\begin{align*}
a_i
&=\dist_{\adjunctionSpace}(p_i,N)\\
&=\inf_{n\in N}\adjunctionMetric(p_i,n)\\
&=\inf_{n\in N}\inf_{u\in\attachmentSet}
\bigl(
\norm{x_i-u}_{E}
+\norm{\attachmentMap(u)-n}_{N}
\bigr)\\
&=\inf_{u\in\attachmentSet}\norm{x_i-u}_{E}\\
&=\dist_E(x_i,\attachmentSet).
\end{align*}
For the penultimate equality,
the expression inside the two infima is at least
$\norm{x_i-u}_{E}$.
Conversely,
for each
$u\in\attachmentSet$,
the choice
$n=\attachmentMap(u)$
makes its second term zero.
The form of
$\attachmentSet$
shows that this distance is attained either at a point of
$M\times\set{0}$
or at
$(a,\attachmentHeight)$.
Put
\[
b_i=(m_i,0)
\qquad\text{and}\qquad
t=(a,\attachmentHeight).
\]
Since
$\dist_E(x_i,M\times\set{0})=\abs{s_i}$
is attained at
$b_i$,
there is
$u_i\in\set{b_i,t}$
such that
\[
a_i=\norm{x_i-u_i}_{E}.
\]
Let
\[
\minimizingPath{i}=[x_i,u_i]_E
\]
be the corresponding affine segment in
$E$.
Its length is
$a_i$.

If
$u_0=u_1=t$,
then
\[
\minimizingPath{0}*\overline{\minimizingPath{1}}
\]
is an admissible polygonal path from
$p_0$
to
$p_1$
of length
$a_0+a_1<d$.
This contradicts
$\adjunctionMetric(p_0,p_1)=d$.

The two minimizers cannot have different forms either.
After interchanging the indices if necessary,
assume that
$u_0=b_0$
and
$u_1=t$.
The polygonal path in
$E$
given by
\[
\mixedPath
=[b_0,x_0]_E*[x_0,x_1]_E*[x_1,t]_E
\]
joins
$b_0$
to
$t$
and has length
\[
\pathLength(\mixedPath)
=a_0+d+a_1
<2d
\le 2r.
\]
On the other hand,
\[
\norm{b_0-t}_{E}
=\norm{m_0-a}_{M}+\attachmentHeight
\ge \attachmentHeight
>4r.
\]
This contradicts the triangle inequality in
$E$.
Consequently,
$u_i=b_i$
for
$i\in\set{0,1}$,
and hence
$a_i=\abs{s_i}$.
In particular,
\begin{equation}\label{eq:small-vertical-legs}
\abs{s_0}+\abs{s_1}<d\le r.
\end{equation}

For
$\abs{s}<r$
and
$n\in N$,
the adjunction formula reduces to
\begin{equation}\label{eq:collar-distance}
\adjunctionMetric(\adjunctionInclusion(m,s),n)
=\abs{s}+\norm{V(m)-n}_{N}.
\end{equation}
To see this,
consider the base route
\[
\baseRoute{m}{s}{n}
=[(m,s),(m,0)]_E*[V(m),n]_N
\]
and the top route
\[
\topRoute{m}{s}{n}
=[(m,s),t]_E*[y,n]_N.
\]
The base route has length
\[
\pathLength(\baseRoute{m}{s}{n})
=\abs{s}+\norm{V(m)-n}_{N}.
\]
Equation
\eqref{eq:global-one-lipschitz}
shows that a base route through
$(u,0)$
has length at least
\[
\abs{s}+\norm{m-u}_{M}+\norm{V(u)-n}_{N}
\ge \pathLength(\baseRoute{m}{s}{n}).
\]
The top route satisfies
\[
\pathLength(\topRoute{m}{s}{n})
\ge
\norm{m-a}_{M}
+\attachmentHeight
-\abs{s}
+\norm{y-n}_{N}.
\]
The triangle inequality and the definition of
$\attachmentGap$
also imply
\[
\pathLength(\baseRoute{m}{s}{n})
\le
\abs{s}
+\norm{m-a}_{M}
+\attachmentGap
+\norm{y-n}_{N}.
\]
It follows that
\[
\pathLength(\topRoute{m}{s}{n})
-\pathLength(\baseRoute{m}{s}{n})
\ge
\attachmentHeight-\attachmentGap-2\abs{s}
>0
\]
because
$\abs{s}<r$
and
$\attachmentHeight>\attachmentGap+4r$.
Thus no route through
$t$
is shorter than
$\baseRoute{m}{s}{n}$,
which proves
\eqref{eq:collar-distance}.

Since
$\norm{m_0-m_1}_{M}\le d\le r$,
equation
\eqref{eq:short-scale}
implies
\[
\norm{V(m_0)-V(m_1)}_{N}
=\norm{m_0-m_1}_{M}.
\]
The Hahn--Banach theorem provides
$\targetFunctional\in \dual{N}$
with
$\norm{\targetFunctional}_{\dual{N}}\le 1$
such that
\[
\targetFunctional(V(m_0)-V(m_1))
=\norm{V(m_0)-V(m_1)}_{N}
=\norm{m_0-m_1}_{M}.
\]
Put
$d_s=\abs{s_0-s_1}$
and set
\[
c_0=\min\set{d_s,\abs{s_0}}
\qquad\text{and}\qquad
c_1=c_0-d_s.
\]
We have
$0\le c_0\le\abs{s_0}$
and
\[
\abs{c_1}=d_s-c_0
=\max\set{d_s-\abs{s_0},0}
\le\abs{s_1}.
\]
The last inequality follows from
$d_s=\abs{s_0-s_1}\le\abs{s_0}+\abs{s_1}$.
Thus
$c_0-c_1=d_s$.
On
$N\cup\set{p_0,p_1}$,
set
\[
f|_N=\targetFunctional,
\qquad
f(p_i)=\targetFunctional(V(m_i))+c_i.
\]
For
$n\in N$,
the  formula
\eqref{eq:collar-distance}
implies
\begin{align*}
\abs{f(p_i)-f(n)}
&\le \norm{V(m_i)-n}_{N}+\abs{c_i}\\
&\le \norm{V(m_i)-n}_{N}+\abs{s_i}\\
&=\adjunctionMetric(p_i,n).
\end{align*}
Its difference on the selected pair is
\[
f(p_0)-f(p_1)
=\norm{m_0-m_1}_{M}+\abs{s_0-s_1}
=d.
\]
Hence
$f$
is 1-Lipschitz on
$N\cup\set{p_0,p_1}$.
Retain the symbol
$f$
for its McShane extension to
$\adjunctionSpace$.
It is an admissible dual functional.
By
\eqref{eq:relative-dual-quotient-action},
it satisfies
\[
\relativeDualFunctional{\targetFunctional}{f}
\bigl(\relativeEmbedding(p_0)-\relativeEmbedding(p_1)\bigr)
=f(p_0)-f(p_1)
=d.
\]
Together with the 1-Lipschitz property of
$\relativeEmbedding\circ\adjunctionInclusion$,
the two cases prove
\eqref{eq:relative-envelope-short}.
\end{proof}

The two auxiliary lemmas now supply the coordinates needed for the successor mechanism.
The next lemma adds a coordinate which separates points and constructs the retraction used at limit stages.

\begin{lemma}\label{lem:protected-extension}
Let
$r,L>0$,
and let
\[
\mathcal A
=\bigl(
\tripleSource{\mathcal A},
\tripleTarget{\mathcal A},
\tripleMap{\mathcal A}
\bigr)
\]
be a uniformly locally isometric triple at scale
$r$.
For every
\[
y\in
\tripleTarget{\mathcal A}
\setminus
\tripleMap{\mathcal A}(\tripleSource{\mathcal A}),
\]
there exist a triple
$\mathcal B$
and maps
$\protectedSourceEmbedding{\mathcal A}{\mathcal B}$,
$\protectedTargetEmbedding{\mathcal A}{\mathcal B}$,
and
$\protectedExtensionRetraction{\mathcal A}{\mathcal B}$
which make
$\mathcal B$
an
$L$-protected extension of
$\mathcal A$
at
$y$.
\end{lemma}

\begin{proof}
For the construction,
write
\[
M=\tripleSource{\mathcal A},
\qquad
N=\tripleTarget{\mathcal A},
\qquad
V=\tripleMap{\mathcal A}.
\]
Fix
$y\in N\setminus V(M)$
and
$a\in M$.
Put
$\attachmentGap=\norm{V(a)-y}_{N}$
and choose
$\attachmentHeight>\attachmentGap+L+4r$.
Since
$L>0$,
this choice also satisfies the height condition in
\Cref{lem:metric-adjunction}.
Apply
\Cref{lem:metric-adjunction}
with these choices,
and retain its notation
$E$,
$\sourceEmbedding$,
$\attachmentSet$,
$\attachmentMap$,
$\adjunctionSpace$,
$\adjunctionInclusion$,
and
$\adjunctionTargetEmbedding$.
Apply
\Cref{lem:relative-envelope}
to the resulting adjunction
and retain its notation
$\relativeEnvelope$,
$\relativeQuotientMap$,
$\oldTargetEmbedding$,
and
$\relativeEmbedding$.

The copy of
$N$
is closed in
$\adjunctionSpace$.
Indeed,
for
$x\notin \attachmentSet$,
we have
$\adjunctionMetric(\adjunctionInclusion(x),N)=\dist(x,\attachmentSet)>0.$
Collapse
$N$
to one point and let
\[
\metricQuotientMap\colon\adjunctionSpace\to\quotientSpace=\adjunctionSpace/N
\]
be the quotient map.
Denote the collapsed point by
$\quotientBasepoint=\metricQuotientMap(0)$.
Its metric is
\[
\quotientMetric(\metricQuotientMap(p_0),\metricQuotientMap(p_1))
=\min\set{
\adjunctionMetric(p_0,p_1),
\adjunctionMetric(p_0,N)+\adjunctionMetric(p_1,N)
}.
\]
By
\cite[Proposition~2.6]{ishiki2023factorization},
this formula defines a metric,
and
$\metricQuotientMap$
is 1-Lipschitz and injective on
$\adjunctionSpace\setminus N$.
For
$z\in \quotientSpace$,
write
\[
\distanceFunction{z}\colon\quotientSpace\to\R,
\qquad
\distanceFunction{z}(u)=\quotientMetric(z,u)
\]
for
$u\in \quotientSpace$,
and let
$\boundedContinuous{\quotientSpace}$
be the Banach space of bounded continuous real functions on
$\quotientSpace$
with the supremum norm.
The map
$\kuratowskiEmbedding\colon \quotientSpace\to \boundedContinuous{\quotientSpace}$
defined by
$\kuratowskiEmbedding(z)=\distanceFunction{z}-\distanceFunction{\quotientBasepoint}$
is an isometric embedding and satisfies
$\kuratowskiEmbedding(\quotientBasepoint)=0$
by
\cite[Theorem~2.6]{ishiki2026interpolation}.
Let
\[
\kuratowskiSpan=\closedLinearSpan{\kuratowskiEmbedding(\quotientSpace)}
\subset \boundedContinuous{\quotientSpace}.
\]
Since
$\kuratowskiEmbedding(\quotientSpace)\subset\kuratowskiSpan$,
we henceforth regard
$\kuratowskiEmbedding$
as a map from
$\quotientSpace$
to
$\kuratowskiSpan$.
Define
$\injectivityCoordinate\colon \adjunctionSpace\to \kuratowskiSpan$
by
$\injectivityCoordinate=\kuratowskiEmbedding\circ\metricQuotientMap$.
The map
$\injectivityCoordinate$
is 1-Lipschitz,
vanishes on
$N$,
and separates two points unless both belong to
$N$.
Set
\[
F=\inftySum{\relativeEnvelope}{\kuratowskiSpan}.
\]
Define
\[
\targetEmbedding\colon N\to F,
\qquad
\targetEmbedding(n)=(\oldTargetEmbedding(n),0),
\]
and
\[
\combinedEmbedding\colon\adjunctionSpace\to F,
\qquad
\combinedEmbedding(p)=(\relativeEmbedding(p),\injectivityCoordinate(p)).
\]
The forward maps are summarized by
\begin{equation}\label{eq:forward-map-configuration}
\begin{gathered}
\adjunctionSpace
\xrightarrow{\ \metricQuotientMap\ }
\quotientSpace
\xrightarrow{\ \kuratowskiEmbedding\ }
\kuratowskiSpan,
\qquad
\injectivityCoordinate=\kuratowskiEmbedding\circ\metricQuotientMap,
\\
\combinedEmbedding=(\relativeEmbedding,\injectivityCoordinate)
\colon\adjunctionSpace\to\inftySum{\relativeEnvelope}{\kuratowskiSpan},
\qquad
\extendedMap=\combinedEmbedding\circ\adjunctionInclusion.
\end{gathered}
\end{equation}
The first coordinate
$\relativeEmbedding$
preserves the required short distances.
The second coordinate
$\injectivityCoordinate$
separates points not already separated inside the old target.
The map
$\combinedEmbedding$
is 1-Lipschitz and injective.
Equation
\eqref{eq:relative-maps-agree}
and the fact that
$\injectivityCoordinate$
vanishes on
$N$
show that
\begin{equation}\label{eq:combined-map-on-old-target}
\combinedEmbedding\circ\adjunctionTargetEmbedding=\targetEmbedding.
\end{equation}
The first coordinate and
\eqref{eq:relative-envelope-short}
show that
$\combinedEmbedding\circ\adjunctionInclusion$
preserves every distance at most
$r$.
By
\eqref{eq:forward-map-configuration},
$\extendedMap$
is injective.
\begin{equation}\label{diag:extension-square}
\begin{tikzcd}[column sep=large,row sep=large]
M \arrow[r,"\sourceEmbedding"] \arrow[d,"V"']
& E \arrow[d,"\extendedMap"]\\
N \arrow[r,"\targetEmbedding"']
& F.
\end{tikzcd}
\end{equation}
The square commutes because
$\adjunctionInclusion\circ\sourceEmbedding
=\adjunctionTargetEmbedding\circ V$,
and
equation
\eqref{eq:combined-map-on-old-target}
then implies commutativity.
The distinguished point at height
$\attachmentHeight$
also satisfies
\[
\extendedMap(a,\attachmentHeight)
=\combinedEmbedding(\adjunctionTargetEmbedding(y))
=\targetEmbedding(y).
\]
Set
\[
\mathcal B=(E,F,\extendedMap),
\qquad
\protectedSourceEmbedding{\mathcal A}{\mathcal B}=\sourceEmbedding,
\qquad
\protectedTargetEmbedding{\mathcal A}{\mathcal B}=\targetEmbedding.
\]
The construction so far proves
\ref{item:extension-scale}--\ref{item:extension-hit}.

It remains to prove
\ref{item:extension-retraction}
and
\ref{item:extension-recovery}.
Let
$\cutoffMap\colon\R\to[0,1]$
vanish on
$({-\infty},L]$,
be affine on
$[L,\attachmentHeight]$,
and equal one on
$[\attachmentHeight,\infty)$.
Define
\[
\sourceRetraction\colon E\to N,
\qquad
\sourceRetraction(m,s)
=V(m)+\cutoffMap(s)(y-V(a)).
\]
Since
\[
\Lip(\cutoffMap)=\frac{1}{\attachmentHeight-L}
\quad\text{and}\quad
\frac{\attachmentGap}{\attachmentHeight-L}<1,
\]
we have
\[
\begin{split}
&\norm{\sourceRetraction(m_0,s_0)-\sourceRetraction(m_1,s_1)}_{N}\\
&\quad\le
\norm{V(m_0)-V(m_1)}_{N}
+\attachmentGap\abs{\cutoffMap(s_0)-\cutoffMap(s_1)}\\
&\quad\le
\norm{m_0-m_1}_{M}
+\frac{\attachmentGap}{\attachmentHeight-L}\abs{s_0-s_1}\\
&\quad\le
\norm{m_0-m_1}_{M}+\abs{s_0-s_1}.
\end{split}
\]
Therefore,
the map
$\sourceRetraction\colon E\to N$
is 1-Lipschitz.
It agrees with
$\attachmentMap$
on
$\attachmentSet$.
Together with the identity on
$N$,
it descends to a pointed 1-Lipschitz retraction
$\adjunctionRetraction\colon \adjunctionSpace\to N$.
By construction,
\begin{equation}\label{eq:adjunction-recovery-maps}
\adjunctionRetraction\circ\adjunctionInclusion=\sourceRetraction,
\qquad
\adjunctionRetraction\circ\adjunctionTargetEmbedding=\operatorname{id}_N.
\end{equation}

By the universal property of
$\freeSpace{\adjunctionSpace}$,
the map
$\adjunctionRetraction$
has a unique 1-Lipschitz linear extension
\[
\linearizedAdjunctionRetraction\colon\freeSpace{\adjunctionSpace}\to N.
\]
It satisfies
\[
\linearizedAdjunctionRetraction(\dirac{p})=\adjunctionRetraction(p)
\]
for every
$p\in\adjunctionSpace$.
Define
\[
\quotientRetractionLift\colon\oneSum{N}{\freeSpace{\adjunctionSpace}}\to N,
\qquad
\quotientRetractionLift(n,\mu)=n+\linearizedAdjunctionRetraction(\mu).
\]
For every
$n\in N$
and
$\mu\in\freeSpace{\adjunctionSpace}$,
we have
\[
\norm{\quotientRetractionLift(n,\mu)}_{N}
\le \norm{n}_{N}+\norm{\linearizedAdjunctionRetraction(\mu)}_{N}
\le \norm{n}_{N}+\norm{\mu}_{\freeSpace{\adjunctionSpace}}.
\]
Thus
$\quotientRetractionLift$
is 1-Lipschitz.
Since
$\adjunctionRetraction(n)=n$
for
$n\in N$,
we have
\[
\quotientRetractionLift(n,-\dirac{n})
=n-\linearizedAdjunctionRetraction(\dirac{n})
=n-\adjunctionRetraction(n)
=0.
\]
It follows that
$\gluingSubspace\subset\ker \quotientRetractionLift$.
Consequently,
$\quotientRetractionLift$
induces a 1-Lipschitz linear map
\[
\relativeRetraction\colon \relativeEnvelope\to N,
\qquad
\relativeRetraction(\relativeQuotientMap(n,\mu))
=n+\linearizedAdjunctionRetraction(\mu).
\]
For
$n\in N$,
we have
\[
\relativeRetraction(\oldTargetEmbedding(n))=n.
\]
Hence
$\relativeRetraction$
is a linear retraction onto the old target.
Moreover,
the defining formula yields
\begin{equation}\label{eq:relative-recovery-map}
\relativeRetraction\circ\relativeEmbedding=\adjunctionRetraction.
\end{equation}

Define the first coordinate projection by
\[
\firstCoordinateProjection\colon F\to\relativeEnvelope,
\qquad
\firstCoordinateProjection(b,g)=b.
\]
Set
\[
\protectedRetraction
=\relativeRetraction\circ\firstCoordinateProjection
\colon F\to N.
\]
This is a 1-Lipschitz linear map,
and
$\protectedRetraction\circ\targetEmbedding=\operatorname{id}_N$.
The named maps now fit into the commutative recovery diagram
\begin{equation}\label{diag:recovery-maps}
\begin{tikzcd}[column sep=large,row sep=large]
E
  \arrow[r,"\adjunctionInclusion"]
  \arrow[d,"\sourceRetraction"']
& \adjunctionSpace
  \arrow[r,"\combinedEmbedding"]
  \arrow[d,"\adjunctionRetraction"']
& F
  \arrow[d,"\protectedRetraction"]\\
N \arrow[r,equal]
& N \arrow[r,equal]
& N.
\end{tikzcd}
\end{equation}
The left square is
\eqref{eq:adjunction-recovery-maps}.
The right square follows from
\eqref{eq:relative-recovery-map}
because
$\firstCoordinateProjection\circ\combinedEmbedding=\relativeEmbedding$.
Consequently,
\[
(\protectedRetraction\circ\extendedMap)(m,s)
=\sourceRetraction(m,s)
=V(m)+\cutoffMap(s)(y-V(a)).
\]
If
$\abs{s}\le L$,
then
$\cutoffMap(s)=0$,
and therefore
\[
(\protectedRetraction\circ\extendedMap)(m,s)=V(m).
\]
Finally,
set
\[
\protectedExtensionRetraction{\mathcal A}{\mathcal B}
=\protectedRetraction.
\]
The preceding identities prove
\ref{item:extension-retraction}
and
\ref{item:extension-recovery},
which completes the proof.
\end{proof}

\section{Coherent limit stages}\label{sec:coherent-limit-stages}
We record how the retractions in
\Cref{lem:protected-extension}
control completed increasing unions.
Preservation at the fixed distance scale alone does not ensure that the extension to a completed increasing union remains injective.
The retractions supply the missing coordinate recovery.

Consider a transfinite chain
$\set{(\sourceStage{\alpha},\targetStage{\alpha},\stageMap{\alpha})}_{\alpha<\tau}$
of Banach spaces and injective maps.
For
$\alpha\le\beta<\tau$,
let
\[
\sourceStageEmbedding{\alpha}{\beta}
\colon\sourceStage{\alpha}\to\sourceStage{\beta}
\quad\text{and}\quad
\targetStageEmbedding{\alpha}{\beta}
\colon\targetStage{\alpha}\to\targetStage{\beta}
\]
be linear isometric embeddings.
These stage embeddings are required to satisfy
\[
\sourceStageEmbedding{\alpha}{\alpha}
=\operatorname{id}_{\sourceStage{\alpha}},
\qquad
\sourceStageEmbedding{\beta}{\gamma}
\circ\sourceStageEmbedding{\alpha}{\beta}
=\sourceStageEmbedding{\alpha}{\gamma}
\]
and
\[
\targetStageEmbedding{\alpha}{\alpha}
=\operatorname{id}_{\targetStage{\alpha}},
\qquad
\targetStageEmbedding{\beta}{\gamma}
\circ\targetStageEmbedding{\alpha}{\beta}
=\targetStageEmbedding{\alpha}{\gamma}
\]
whenever
$\alpha\le\beta\le\gamma<\tau$.
We also require the square
\begin{equation}\label{diag:stage-compatibility}
\begin{tikzcd}[column sep=large,row sep=large]
\sourceStage{\alpha}
  \arrow[r,"\sourceStageEmbedding{\alpha}{\beta}"]
  \arrow[d,"\stageMap{\alpha}"']
& \sourceStage{\beta}
  \arrow[d,"\stageMap{\beta}"]\\
\targetStage{\alpha}
  \arrow[r,"\targetStageEmbedding{\alpha}{\beta}"']
& \targetStage{\beta}
\end{tikzcd}
\end{equation}
to commute.
Thus
\begin{equation}\label{eq:stage-map-compatibility}
\stageMap{\beta}\circ\sourceStageEmbedding{\alpha}{\beta}
=\targetStageEmbedding{\alpha}{\beta}\circ\stageMap{\alpha}.
\end{equation}
We use the stage embeddings to identify each stage with its image in every later stage.
Under these identifications,
the source spaces and target spaces are increasing chains and the stage maps extend one another.
Assume that every
$\stageMap{\alpha}\colon \sourceStage{\alpha}\to \targetStage{\alpha}$
preserves distances at most a fixed
$r>0$.
Fix
$L>0$.
Whenever
\Cref{lem:protected-extension}
is applied at a successor stage
$\alpha+1$,
set
\[
\sourceStage{\alpha+1}=\oneSum{\sourceStage{\alpha}}{\R},
\]
take
$\sourceStageEmbedding{\alpha}{\alpha+1}$
and
$\targetStageEmbedding{\alpha}{\alpha+1}$
to be the embeddings
$\sourceEmbedding$
and
$\targetEmbedding$
from that lemma,
and let
\[
\stageRetraction{\alpha}{\alpha+1}\colon \targetStage{\alpha+1}\to \targetStage{\alpha}
\]
be the 1-Lipschitz linear retraction obtained there.
It satisfies
\begin{equation}\label{eq:successor-coordinate-recovery}
(\stageRetraction{\alpha}{\alpha+1}\circ\stageMap{\alpha+1})(m,s)
=\stageMap{\alpha}(m)
\end{equation}
whenever
$\abs{s}\le L$.
At a successor stage where no extension is required,
set
\[
\sourceStage{\alpha+1}=\sourceStage{\alpha},
\qquad
\targetStage{\alpha+1}=\targetStage{\alpha},
\qquad
\stageMap{\alpha+1}=\stageMap{\alpha},
\]
and use the identity as
$\sourceStageEmbedding{\alpha}{\alpha+1}$,
$\targetStageEmbedding{\alpha}{\alpha+1}$,
and
$\stageRetraction{\alpha}{\alpha+1}$.

Retractions to earlier stages are defined by composition at successors.
They satisfy
\begin{equation}\label{eq:retraction-coherence}
\stageRetraction{\alpha}{\beta}\circ\targetStageEmbedding{\alpha}{\beta}
=\operatorname{id}_{\targetStage{\alpha}},
\qquad
\stageRetraction{\alpha}{\gamma}\circ\targetStageEmbedding{\beta}{\gamma}
=\stageRetraction{\alpha}{\beta}
\end{equation}
whenever
$\alpha\le\beta\le\gamma$.

\begin{proposition}\label{prop:coherent-limit-stage}
Let
$\lambda$
be a limit stage of the construction above.
Assume that every stage map below
$\lambda$
is injective and preserves every distance at most
$r$.
Let
$\sourceStage{\lambda}$
and
$\targetStage{\lambda}$
be the completions of the corresponding algebraic increasing unions.
Then the coherent union of the stage maps extends uniquely to an injective 1-Lipschitz map
\[
\stageMap{\lambda}\colon\sourceStage{\lambda}\to\targetStage{\lambda}
\]
which extends every earlier stage map and preserves every distance at most
$r$.
For each
$\alpha<\lambda$,
the coherent retractions extend uniquely to a 1-Lipschitz linear retraction
\[
\stageRetraction{\alpha}{\lambda}
\colon\targetStage{\lambda}\to\targetStage{\alpha}
\]
and continue to satisfy
\eqref{eq:retraction-coherence}.

More precisely,
call the transition from stage
$\eta$
to stage
$\eta+1$
active when it applies
\Cref{lem:protected-extension}.
For
$\delta\le\lambda$,
define
\[
\extensionStages{\delta}
=\set{
\eta<\delta
\mid
\text{the transition from }\eta\text{ to }\eta+1\text{ is active}
}.
\]
Thus each
$\eta\in\extensionStages{\delta}$
indexes the new
$\R$
coordinate introduced in that transition,
while an idle transition contributes no coordinate.
We have
$\extensionStages{\gamma}=\extensionStages{\lambda}\cap\gamma$
for
$\gamma<\lambda$.
Under the stage embeddings,
the source spaces have the canonical coordinate representations
\[
\sourceStage{\gamma}
=\oneSum{\sourceStage{0}}{\ellOneOf{\extensionStages{\gamma}}}
\]
for
$\gamma\le\lambda$.
In particular,
\[
\sourceStage{\lambda}
=\oneSum{\sourceStage{0}}{\ellOneOf{\extensionStages{\lambda}}}.
\]
For
$\gamma<\lambda$,
let
\[
\stageProjection{\gamma}
\colon\sourceStage{\lambda}\to\sourceStage{\gamma}
\]
be the projection onto stage
$\gamma$,
defined by
\[
\stageProjection{\gamma}
\bigl(m,(s_\eta)_{\eta\in\extensionStages{\lambda}}\bigr)
=\bigl(m,(s_\eta)_{\eta\in\extensionStages{\gamma}}\bigr).
\]
If
\[
x=(m,(s_\eta)_{\eta\in\extensionStages{\lambda}})
\]
and
$\gamma$
is chosen so that
$\eta<\gamma$
whenever
$\abs{s_\eta}>L$,
then
\begin{equation}\label{eq:tail-coordinate-recovery}
\bigl(\stageRetraction{\gamma}{\lambda}\circ\stageMap{\lambda}\bigr)(x)
=\stageMap{\gamma}(\stageProjection{\gamma}(x)).
\end{equation}
\end{proposition}

\begin{proof}
Fix a limit stage
$\lambda$
and assume by transfinite induction that the recovery identity has been established at every earlier limit stage.
For fixed
$\alpha<\lambda$,
the maps
$\stageRetraction{\alpha}{\beta}$
agree on overlaps by
\eqref{eq:retraction-coherence}.
They therefore define a 1-Lipschitz linear map on the algebraic union by
\[
\stageRetraction{\alpha}{\lambda}
\bigl(\targetStageEmbedding{\beta}{\lambda}(n)\bigr)
=\stageRetraction{\alpha}{\beta}(n)
\]
for
$\alpha\le\beta<\lambda$
and
$n\in\targetStage{\beta}$.
This map extends uniquely to a 1-Lipschitz linear retraction
$\stageRetraction{\alpha}{\lambda}\colon \targetStage{\lambda}\to \targetStage{\alpha}$.
The extensions continue to satisfy
\eqref{eq:retraction-coherence}.

At every transition indexed by
$\eta\in\extensionStages{\lambda}$,
the protected extension replaces the current source
$\sourceStage{\eta}$
by
$\sourceStage{\eta+1}=\oneSum{\sourceStage{\eta}}{\R}$,
and the stage embedding satisfies
\[
\sourceStageEmbedding{\eta}{\eta+1}(m)=(m,0)
\]
for
$m$
in
$\sourceStage{\eta}$.
An idle transition leaves the source unchanged.
Induction over the stages below
$\lambda$
therefore identifies
$\sourceStage{\alpha}$
with
$\oneSum{\sourceStage{0}}{\ellOneOf{\extensionStages{\alpha}}}$
inside
$\oneSum{\sourceStage{0}}{\ellOneOf{\extensionStages{\lambda}}}$
by extending later coordinates by zero.
Consequently,
\[
\oneSum{\sourceStage{0}}{\finitelySupported{\extensionStages{\lambda}}}
\subseteq
\bigcup_{\alpha<\lambda}\sourceStage{\alpha}
\subseteq
\oneSum{\sourceStage{0}}{\ellOneOf{\extensionStages{\lambda}}},
\]
where
$\finitelySupported{\extensionStages{\lambda}}$
denotes the finitely supported scalar families on
$\extensionStages{\lambda}$
with the
$\ellOne$
norm.
The first inclusion holds because every finite subset of
$\extensionStages{\lambda}$
is contained in
$\extensionStages{\gamma}$
for some
$\gamma<\lambda$.
Since the completion of
$\finitelySupported{\extensionStages{\lambda}}$
in this norm is
$\ellOneOf{\extensionStages{\lambda}}$,
the algebraic increasing union is dense in the right hand side.
Therefore,
the completed source has the canonical form
\begin{equation}\label{eq:limit-domain}
\sourceStage{\lambda}
=\oneSum{\sourceStage{0}}{\ellOneOf{\extensionStages{\lambda}}}.
\end{equation}
The coherent union of the maps
$\stageMap{\alpha}$
is 1-Lipschitz by
\eqref{eq:global-one-lipschitz}.
It therefore extends uniquely to a 1-Lipschitz map
\[
\stageMap{\lambda}\colon \sourceStage{\lambda}\to \targetStage{\lambda}.
\]
For every
$\alpha<\lambda$,
the extension satisfies
\[
\stageMap{\lambda}\circ\sourceStageEmbedding{\alpha}{\lambda}
=\targetStageEmbedding{\alpha}{\lambda}\circ\stageMap{\alpha}.
\]

Common finite coordinate truncations show that
$\stageMap{\lambda}$
preserves every distance at most
$r$.
For a finite set
$F\subset\extensionStages{\lambda}$,
let
\[
\coordinateTruncation{F}\colon\sourceStage{\lambda}\to\sourceStage{\lambda}
\]
be the coordinate projection which retains the
$\sourceStage{0}$
component and the coordinates indexed by
$F$.
Its operator norm satisfies
\[
\norm{\coordinateTruncation{F}}_{\mathrm{op}}\le 1.
\]
If
$x_0,x_1\in\sourceStage{\lambda}$
and
$\norm{x_0-x_1}_{\sourceStage{\lambda}}\le r$,
then
$\coordinateTruncation{F}(x_0)$
and
$\coordinateTruncation{F}(x_1)$
belong to a common earlier stage and remain at distance at most
$r$.
Their distance is therefore preserved by that stage map and hence by
$\stageMap{\lambda}$.
As the finite sets
$F$
increase,
$\coordinateTruncation{F}(x_i)$
converges to
$x_i$
for
$i\in\set{0,1}$.
Continuity now yields equality at the limit.

Fix
$x\in \sourceStage{\lambda}$
and write
\[
x=(m,(s_\alpha)_{\alpha\in\extensionStages{\lambda}})
\in\oneSum{\sourceStage{0}}{\ellOneOf{\extensionStages{\lambda}}}.
\]
Choose
$\gamma<\lambda$
so that
$\alpha<\gamma$
whenever
$\alpha\in\extensionStages{\lambda}$
and
$\abs{s_\alpha}>L$.
There are only finitely many such coordinates because the tail in
\eqref{eq:limit-domain}
belongs to
$\ellOneOf{\extensionStages{\lambda}}$.

For a finite set
$F\subset\extensionStages{\lambda}\cap[\gamma,\lambda)$,
let
\[
\stageTailTruncation{\gamma}{F}\colon
\sourceStage{\lambda}\to\sourceStage{\lambda}
\]
be the coordinate projection which retains the
$\sourceStage{0}$
component and the coordinates indexed by
$\extensionStages{\gamma}\cup F$.
Put
$x_F=\stageTailTruncation{\gamma}{F}(x)$.
Choose
$\delta$
with
$\gamma\le\delta<\lambda$
so that
$x_F\in\sourceStage{\delta}$.
By compatibility of the stage maps and
\eqref{eq:retraction-coherence},
we have
\[
\stageMap{\lambda}(x_F)
=\targetStageEmbedding{\delta}{\lambda}(\stageMap{\delta}(x_F))
\quad\text{and}\quad
\stageRetraction{\gamma}{\lambda}
\circ\targetStageEmbedding{\delta}{\lambda}
=\stageRetraction{\gamma}{\delta}.
\]
Every coordinate of
$x_F$
added after stage
$\gamma$
has absolute value at most
$L$.
The successor identity
\eqref{eq:successor-coordinate-recovery},
the identity maps at idle stages,
and the induction hypothesis at earlier limit stages therefore yield the complete map calculation
\begin{align}
\bigl(\stageRetraction{\gamma}{\lambda}\circ\stageMap{\lambda}\bigr)(x_F)
&=\bigl(\stageRetraction{\gamma}{\lambda}
\circ\targetStageEmbedding{\delta}{\lambda}
\circ\stageMap{\delta}\bigr)(x_F)\notag\\
&=\bigl(\stageRetraction{\gamma}{\delta}
\circ\stageMap{\delta}\bigr)(x_F)\notag\\
&=\stageMap{\gamma}(\stageProjection{\gamma}(x)).
\label{eq:finite-stage-recovery}
\end{align}
As
$F$
increases through the finite subsets of
$\extensionStages{\lambda}\cap[\gamma,\lambda)$,
the points
$x_F$
converge to
$x$.
Continuity therefore yields
\eqref{eq:tail-coordinate-recovery}.

Assume that
\[
\stageMap{\lambda}(x_0)=\stageMap{\lambda}(x_1).
\]
Write
\[
x_i=(m_i,(s_{i,\eta})_{\eta\in\extensionStages{\lambda}})
\]
for
$i\in\set{0,1}$.
Choose
$\beta<\lambda$
so that
$\eta<\beta$
whenever
$\eta\in\extensionStages{\lambda}$
and
\[
\max\set{\abs{s_{0,\eta}},\abs{s_{1,\eta}}}>L.
\]
For every
$\gamma$
with
$\beta\le\gamma<\lambda$,
apply
$\stageRetraction{\gamma}{\lambda}$
to the equality and use
\eqref{eq:tail-coordinate-recovery}.
We obtain
\[
\stageMap{\gamma}(\stageProjection{\gamma}(x_0))
=\stageMap{\gamma}(\stageProjection{\gamma}(x_1)).
\]
Injectivity at stage
$\gamma$
implies
\[
\stageProjection{\gamma}(x_0)=\stageProjection{\gamma}(x_1).
\]
The projection onto every stage retains the
$\sourceStage{0}$
component.
For each
$\eta\in\extensionStages{\lambda}$,
choose
$\gamma$
with
$\max\set{\beta,\eta+1}\le\gamma<\lambda$.
The equality of the two stage projections then implies
$s_{0,\eta}=s_{1,\eta}$.
Hence
$x_0=x_1$.
\end{proof}

Together with
\Cref{lem:protected-extension},
\Cref{prop:coherent-limit-stage}
shows by transfinite induction that all successor maps and all completed limit maps remain injective and preserve every distance at most
$r$.

\section{The counterexample}\label{sec:counterexample}
We now assemble the counterexample.
The construction has three parts.
We first define the initial map,
then control the size of every stage so that target points can be scheduled,
and finally run the transfinite recursion.

\subsection{The bent seed}

Define the polygonal map
$\bentMap\colon\R\to\bentTarget$
by
\begin{equation}\label{eq:bent-seed}
\bentMap(t)=
\begin{cases}
(t,0), & t\le 0,\\
(0,t), & 0\le t\le 1,\\
(1-t,1), & 1\le t.
\end{cases}
\end{equation}
The image consists of two parallel rays joined by a vertical segment,
so
$\bentMap$
is injective.
If
$\abs{s-t}\le 1/2$,
the parameter interval between
$s$
and
$t$
meets at most one corner.
At either corner the two directions occupy different
$\ellOne$
coordinates.
It follows that
\begin{equation}\label{eq:bent-seed-short}
\norm{\bentMap(s)-\bentMap(t)}_{\bentTarget}=\abs{s-t}
\end{equation}
whenever
$\abs{s-t}\le 1/2$.
On the other hand,
\begin{equation}\label{eq:bent-seed-contraction}
\norm{\bentMap(-1)-\bentMap(2)}_{\bentTarget}=1<3=\abs{-1-2}.
\end{equation}

\subsection{Cardinal bounds and scheduling}

The bent map supplies the required failure of global distance preservation,
but its image is not the whole target.
To add every missing target point later,
we first choose a recursion length larger than the cardinality of every possible stage.
Put
\[
\continuumCardinal=2^{\aleph_0},
\qquad
\stageCardinal=2^{\continuumCardinal},
\qquad
\recursionCardinal=\stageCardinal^+.
\]
The cardinal
$\recursionCardinal$
is regular and uncountable.
We shall construct a chain through all ordinals below
$\recursionCardinal$.

Every stage of the construction will have cardinality at most
$\stageCardinal$.
The initial spaces have this property.
At a successor stage,
the metric adjunction constructed in
\Cref{lem:metric-adjunction}
has cardinality at most
$\stageCardinal$.
Its Lipschitz-free space has density at most
$\stageCardinal$.
The relative Banach envelope
$\relativeEnvelope$
from
\Cref{lem:relative-envelope}
therefore has density at most
$\stageCardinal$.
The metric quotient
$\quotientSpace$
has cardinality at most
$\stageCardinal$,
and its Kuratowski image has closed linear span
$\kuratowskiSpan$
of density at most
$\stageCardinal$.
Thus the direct sum used in
\Cref{lem:protected-extension}
also has density at most
$\stageCardinal$.
A Banach space of density at most
$\stageCardinal$
has cardinality at most
\[
\stageCardinal^{\aleph_0}
=(2^{\continuumCardinal})^{\aleph_0}
=2^{\continuumCardinal}
=\stageCardinal.
\]

If
$\lambda<\recursionCardinal$
is a limit ordinal,
then
$\abs{\lambda}\le\stageCardinal$.
The increasing union at stage
$\lambda$
therefore has cardinality at most
$\stageCardinal$,
and its completion has the same bound.
This proves the stage bound by transfinite induction.

To ensure that every target point is eventually included in the image,
we use a standard bookkeeping argument.
Fix an injection
\begin{equation}\label{eq:schedule}
\scheduleMap\colon\recursionCardinal\times\recursionCardinal\to\recursionCardinal
\end{equation}
such that
$\alpha<\scheduleMap(\alpha,\xi)$
for all
$\alpha,\xi<\recursionCardinal$.
Such an injection can be constructed recursively after well ordering
$\recursionCardinal\times\recursionCardinal$
in order type
$\recursionCardinal$.
At each step fewer than
$\recursionCardinal$
ordinals have been used,
while the tail above the selected first coordinate still has cardinality
$\recursionCardinal$.

\subsection{Proof of the main theorem}

At a successor transition which processes a missing target point,
we use
\Cref{lem:protected-extension}.
An idle successor transition uses the identity maps,
and a limit stage uses
\Cref{prop:coherent-limit-stage}.
Equation
\eqref{eq:stage-map-compatibility}
then makes the stage maps into one well-defined final map.
\begin{proof}[Proof of \Cref{thm:counterexample}]
We construct the stage data by transfinite recursion on
$\beta<\recursionCardinal$.
When stage
$\beta$
is formed,
the data consist of Banach spaces
$\sourceStage{\beta}$
and
$\targetStage{\beta}$,
a map
$\stageMap{\beta}\colon\sourceStage{\beta}\to\targetStage{\beta}$,
and a surjection
\[
\stageEnumeration{\beta}\colon\recursionCardinal\to\targetStage{\beta}.
\]
We retain the constants
\[
r=1/2
\quad\text{and}\quad
L=1
\]
at every stage.
The stage invariant is that
$\stageMap{\beta}$
is injective and preserves every distance at most
$r$.

At stage zero,
set
\[
\sourceStage{0}=\R,
\qquad
\targetStage{0}=\bentTarget,
\qquad
\stageMap{0}=\bentMap,
\]
where
$\bentMap$
is defined by
\eqref{eq:bent-seed}.
The cardinal bound permits us to choose the surjection
\[
\stageEnumeration{0}\colon\recursionCardinal\to \targetStage{0}.
\]

Assume next that stage
$\beta$
has been formed.
Once stage
$\delta$
has been formed,
the pair
$(\delta,\xi)\in\recursionCardinal\times\recursionCardinal$
represents the requirement that the point
$\stageEnumeration{\delta}(\xi)$
eventually belong to the image of a stage map.
The schedule map assigns this
$\recursionCardinal\times\recursionCardinal$-indexed
family of requirements to distinct transition indices below
$\recursionCardinal$.
At the transition from
$\beta$
to
$\beta+1$,
there is at most one pair
$(\delta,\xi)$
with
\[
\beta=\scheduleMap(\delta,\xi)
\]
because
$\scheduleMap$
is injective.
If there is no such pair,
leave the stage unchanged.
If there is such a pair,
then
$\delta<\beta$
by the defining property of
$\scheduleMap$.
The stage embedding therefore defines the scheduled point
\[
y_{\delta,\xi}^{\beta}
:=\targetStageEmbedding{\delta}{\beta}
\bigl(\stageEnumeration{\delta}(\xi)\bigr)
\in\targetStage{\beta}.
\]
When
$y_{\delta,\xi}^{\beta}\notin \stageMap{\beta}(\sourceStage{\beta})$,
the stage invariant permits us to apply
\Cref{lem:protected-extension}
to
$\stageMap{\beta}$
and
$y_{\delta,\xi}^{\beta}$.
When
$y_{\delta,\xi}^{\beta}$
already belongs to the image,
leave the stage unchanged.
After
$\targetStage{\beta+1}$
has been formed,
the cardinal bound permits us to choose
\[
\stageEnumeration{\beta+1}
\colon\recursionCardinal\to\targetStage{\beta+1}.
\]

At a limit ordinal
$\lambda<\recursionCardinal$,
form
$\sourceStage{\lambda}$,
$\targetStage{\lambda}$,
and
$\stageMap{\lambda}$
as in
\Cref{prop:coherent-limit-stage}.
The cardinal bound then permits us to choose
\[
\stageEnumeration{\lambda}
\colon\recursionCardinal\to\targetStage{\lambda}.
\]
This completes the transfinite recursion.

The same recursive clauses prove by transfinite induction that
every
$\stageMap{\beta}$
is injective and preserves every distance at most
$1/2$.
At an active successor transition this follows from
\Cref{lem:protected-extension}.
At an idle transition it is immediate,
and at a limit stage it follows from
\Cref{prop:coherent-limit-stage}.
The stage maps satisfy
\eqref{eq:stage-map-compatibility},
and every stage embedding is a linear isometry.
The cardinal calculation above verifies at each recursive clause that the required enumeration exists.

Define the final spaces by
\[
X=\bigcup_{\alpha<\recursionCardinal}\sourceStage{\alpha},
\qquad
Y=\bigcup_{\alpha<\recursionCardinal}\targetStage{\alpha}.
\]
Define
\[
U\colon X\to Y,
\qquad
U(x)=\stageMap{\alpha}(x)
\quad\text{when }x\in\sourceStage{\alpha}.
\]
Equation
\eqref{eq:stage-map-compatibility}
shows that this definition is independent of the chosen stage
$\alpha$.
The spaces
$X$
and
$Y$
are complete.
Indeed,
the members of a Cauchy sequence
$\set{x_n}_{n\in\Z_{\ge 0}}$
belong to countably many stages.
Their stage indices have supremum below
$\recursionCardinal$
by regularity.
The entire sequence is therefore contained in one earlier Banach stage and converges there.
The same argument applies to the target.

The map
$U$
is injective and preserves every distance at most
$1/2$,
since every pair of source points belongs to a common stage.
It is also surjective.
To see this,
fix
$y\in Y$.
There is
$\delta<\recursionCardinal$
with
$y\in \targetStage{\delta}$,
and there is
$\xi<\recursionCardinal$
with
\[
\stageEnumeration{\delta}(\xi)=y.
\]
Put
$\beta=\scheduleMap(\delta,\xi)$.
Under the stage identifications,
the scheduled point
\[
y_{\delta,\xi}^{\beta}
=\targetStageEmbedding{\delta}{\beta}(y)
\]
is the same element of
$Y$
as
$y$.
At the transition from
$\beta$
to
$\beta+1$,
this point is put into the image if it is still missing.
It remains in the image at every later stage.

If
$p,q\in \closedBall{X}{x}{1/4}$,
then
$\norm{p-q}_{X}\le 1/2$.
Hence
\[
\norm{U(p)-U(q)}_{Y}=\norm{p-q}_{X}
\]
by the fixed short distance property.
Finally,
the points
$-1$
and
$2$
from the initial source remain at distance three,
while their images remain at distance one by
\eqref{eq:bent-seed-contraction}.
Thus
$U$
is not a global isometry.
\end{proof}

The counterexample shows that the local closed ball hypothesis alone is insufficient.
The positive results recalled in the introduction leave the following problem.

\begin{problem}\label{prob:sufficient-conditions}
Find further structural conditions on real Banach spaces
$X$
and
$Y$,
or regularity conditions on a bijection
$U\colon X\to Y$,
which ensure that the local closed ball hypothesis of Scottish Book Problem~155 forces
$U$
to be a global isometry.
\end{problem}

\printbibliography

\end{document}